\documentstyle[11pt,leqno]{article}

\newtheorem{theorem}{Theorem}[section]
\newtheorem{lemma}{Lemma}[section]

\newtheorem{condition}{Condition}[section]
\newtheorem{remark}{Remark}[section]

\catcode`\@=11
\renewcommand{\section}{
         \setcounter{equation}{0}
         \@startsection {section}{1}{\z@}{-3.5ex plus -1ex minus
         -.2ex}{2.3ex plus .2ex}{\normalsize\bf}
}
\renewcommand{\subsection}{
         \@startsection {subsection}{1}{\z@}{-3.5ex plus -1ex minus
         -.2ex}{2.3ex plus .2ex}{\normalsize\bf}
}
\catcode`\@=12

\def\reals{{\rm\vrule depth0ex width.4pt\kern-.08em R}}
\def\bbbz{{\mathchoice {\hbox{$\sf\textstyle Z\kern-0.4em Z$}}
{\hbox{$\sf\textstyle Z\kern-0.4em Z$}}
{\hbox{$\sf\scriptstyle Z\kern-0.3em Z$}}
{\hbox{$\sf\scriptscriptstyle Z\kern-0.2em Z$}}}}

\newcommand{\nc}{\newcommand}
\nc{\W}{{\bf W}}
\nc{\A}{{\bf A}}
\nc{\bL}{{\bf L}}
\nc{\bH}{{\bf H}}
\nc{\C}{{\cal C}}

\def\eq#1{(\ref{e:#1})}

\def\elabel#1{\label{e:#1}}

\begin{document}
\begin{center}
\Large \bf Heavy traffic limit theorems for a queue with
Poisson ON/OFF long-range dependent sources and general
service time distribution
\end{center}
\begin{center}
Wanyang Dai \footnote{Supported by National Natural Science
Foundation of China with grant No. 10371053 and grant No. 10971294.}\\
Department of Mathematics\\
and State Key Laboratory of Novel Software Technology\\
Nanjing University\\
Nanjing 210093, China\\
Email: nan5lu8@netra.nju.edu.cn
\end{center}
\begin{center}
Submitted on 20 June 2008 and revised on 26 April 2011
\end{center}

\vskip 0.1 in
\begin{abstract}

In Internet environment, traffic flow to a link is typically modeled
by superposition of ON/OFF based sources. During each ON-period for
a particular source, packets arrive according to a Poisson process
and packet sizes (hence service times) can be generally distributed.
In this paper, we establish heavy traffic limit theorems to provide
suitable approximations for the system under first-in first-out
(FIFO) and work-conserving service discipline, which state that,
when the lengths of both ON- and OFF-periods are lightly tailed, the
sequences of the scaled queue length and workload processes converge
weakly to short-range dependent reflecting Gaussian processes, and
when the lengths of ON- and/or OFF-periods are heavily tailed with
infinite variance, the sequences converge weakly to either
reflecting fractional Brownian motions (FBMs) or certain type of
long-range dependent reflecting Gaussian processes depending on the
choice of scaling as the number of superposed sources tends to
infinity. Moreover, the sequences exhibit a state space
collapse-{\em like} property when the number of sources is large
enough, which is a kind of extension of the well-known Little's law
for M/M/1 queueing system. Theory to justify the approximations is
based on appropriate heavy traffic conditions which essentially mean
that the service rate closely approaches the arrival rate when the
number of input sources tends to infinity.\\


\end{abstract}

\section{Introduction}

ON/OFF sources are widely used to model voice, video and data
traffics in telecommunication systems (see, e.g., Jain and
Routhier~\cite{jairou:pactra}, Nikolaidis and
Akyildiz~\cite{nikaky:ovesou}, Taqqu {\em et
al.}~\cite{taqwil:profun}, Leland et al.~\cite{leltaq:selsim},
Paxson and Floyd~\cite{paxflo:widare}). In particular, stochastic
modeling of queueing systems with ON/OFF long-range dependent data
has become an active area of research. In contrast to most of the
existing achievements in this field, which are based on fluid models
whose outputs are deterministic with constant (service) rates and
whose inputs are certain types of long-range dependent fluid sources
(e.g. Norros~\cite{nor:stomod}, Massoulie and
Simonian~\cite{massim:larbuf}, Debicki and
Mandjes~\cite{debman:trafbm}, Debicki and
Palmowski~\cite{debpal:flumod}), we will model our queueing system
with general service time distribution and the input as a
superposition of Poisson ON/OFF point processes to better capture
the variation of packet sizes and the behavior of real packet
traffic. Concretely, for a particular source, packets arrive
according to a Poisson process during each ON-period. For such a
source, the corresponding traffic exhibits long-range dependence
(see, for instance, Ryu and Lowen~\cite{ryulow:poipro}) when the
lengths of ON- and/or OFF-periods are heavily tailed with infinite
variance. Besides the assumption on the service time distribution,
our system is further supposed to operate under FIFO and
work-conserving discipline.

A special case of the above queueing model is discussed in Cao and
Ramanan~\cite{caoram:poilim}, where the distributions of ON- and
OFF-periods are assumed to be Pareto and exponential respectively
and packet sizes are supposed to be constant. They show that the
sequence of probabilities that steady state unfinished works exceed
a threshold tend to the corresponding probability assuming Poisson
input process when the number of input sources tends to infinity.
Currently, it is not clear whether their result can be extended to
the more general model as presented above. Furthermore, the
dependence of the convergence rates on various parameters of the
system is not shown in their result, e.g., the relationship between
$\rho^{N}$ (traffic intensity, utilization level) and $N$ (the
number of sources).

Due to the above reasons, we will study our queueing system by
employing some other method. Under heavy traffic conditions
(suitable relationships between $\rho$ and $N$ such that the service
rate closely approaches the arrival rate when $N$ tends to
infinity), we will show that, when the lengths of both ON- and
OFF-periods are lightly tailed, the sequences of the scaled queue
length and workload processes converge weakly to short-range
dependent reflecting Gaussian processes, and when the lengths of ON-
and/or OFF-periods are heavily tailed with infinite variance, the
sequences converge weakly to either reflecting fractional Brownian
motions (FBMs) or certain type of long-range dependent reflecting
Gaussian processes depending on the choice of scaling as the number
of input sources tends to infinity. Moreover, the sequences exhibit
a state space collapse-{\em like} property when $N$ is large enough,
which is a kind of extension of the well-known Little's law for
M/M/1 queueing system.

Our heavy traffic limits set up certain connection between the above
physical queueing systems and some existing fluid queueing models.
For example, for a fluid model with constant
output rate and FBM input, the stationary queue content distribution
is asymptotically Weibullian (e.g., Norros~\cite{nor:stomod},
Massoulie and Simonian~\cite{massim:larbuf}, Husler and
Piterbarg~\cite{huspit:extcer}, and more generally, as
summarized in Whitt~\cite{whi:stopro}), namely, the probability of
exceeding buffer level $b$ is roughly of the form exp($-b^{2(1-H)}$)
if FBM is characterized by Hurst parameter $H$. The result can be
applied to derive corresponding probability for our reflecting FBM
after properly managing parameters.

Concerning heavy traffic limit theorems for queueing systems with
long-range dependent inputs, there are only a few achievements until
now besides the one mentioned above in~\cite{caoram:poilim}. In
Debicki and Mandjes~\cite{debman:trafbm} and Debicki and
Palmowski~\cite{debpal:flumod}, authors studied a fluid queueing
system with constant output rate and a superposition of ON/OFF fluid
input sources. In Konstantopoulos and Lin~\cite{konlin:frabro} and
Majewski~\cite{maj:frabro}, instead of discussing superposition
problem, authors considered a single class and feedforward
multiclass queueing networks with long-range dependent interarrival
and service time sequences respectively. The current limit theorems
are the supplements of these existing results. In justifying our
reflecting FBM approximation, we will adopt the simultaneous limit
regime related to FBM in Mikosch {\em et al.}~\cite{mikres:nettra},
in which both $N$ (the number of sources) and $T$ (the time-scaling
parameter) go to infinity at the same time. This procedure provides
us some convenience in employing some ingredient developed
in~\cite{mikres:nettra} to establish the weak convergence for our
scaled queue length and workload processes.

One last point we wish to mention is that we have employed our
theorem on reflecting Gaussian processes in the current paper to
provide a reasonable interpretation (in Dai~\cite{dai:contru}) to
some well-known large-scale computer and statistical experiments
conducted by Cao {\em et al.}~\cite{caocle:effsta}, Cao and
Ramanan~\cite{caoram:poilim}, which reveal some gapes between their
simulation findings and the existing theory on heavy-tail and long
range dependence. In Dai~\cite{dai:contru}, the author finds out
that all the `heavy-tail' random variables used in computer and
network simulations are truncated versions of their real heavy-tail
counterparts due to the limitations of computer hardware and
softwares, and hence they are not heavily tailed ones. So, by
combining the findings in Dai~\cite{dai:contru} and the theorem in
the current paper, we claim in Dai~\cite{dai:contru} that the
findings in Cao {\em et al.}~\cite{caocle:effsta}, Cao and
Ramanan~\cite{caoram:poilim} are more close to practice and but not
to the mathematical assumptions imposed in their models since their
simulations are computer-based ones.

To be convenient for readers, here we summarize some frequently used
notations and terminologies throughout the paper. First, we recall
the definition of u.o.c. convergence. For a function
$f:[0,\infty)\rightarrow R$ and $t\geq 0$, put
$\|f\|_{t}\equiv\sup_{0\leq s\leq t}|f(s)|$, then a sequence of
functions $f^{n}:[0,\infty)\rightarrow R$ is said to converge
uniformly on compact sets (u.o.c.) to $f$ if for each $t\geq 0$,
$\|f^{n}-f\|_{t}\rightarrow 0$ as $n\rightarrow\infty$. Second, we
use $C_{b}(R)$ to denote the set of all bounded and continuous
functions $f$ and $C(R)$ to denote the set of all continuous
functions over the real number space $R$, which are endowed with the
uniform topology. Third, we use $D_{E}[0,\infty)$ to denote the
Skorohod topological space, i.e., the space of $E$-valued functions
that are right continuous and have left-hand limits, which is
endowed with the Skorohod topology (see, e.g., Ethier an
Kurtz~\cite{ethkur:marpro}, Billingsley~\cite{bil:conpro}). Fourth,
we use i.i.d to denote independent and identically distributed, use
a.s. to denote almost surely, use `$\Rightarrow$' to denote
`converge in distribution' or equivalently `converge weakly', and
use `$\sim$' to denote `equals approximately'.

The rest of this paper is organized as follows. In
Section~\ref{SRBM}, we formulate our model, and in
Section~\ref{limthe}, we present our main theorems and they
are proved in Section~\ref{proofI}.

\section{Queueing model formulation} \label{SRBM}

In this section, we consider a queueing system with general service
time distribution and with $N$ i.i.d. Poisson ON/OFF input sources.
Concretely, a Poisson ON/OFF source $n\in\{1,..,N\}$ consists of
independent strictly alternating ON- and OFF-periods, moreover, it
transmits packets to a server according to a Poisson process with
interarrival time sequence $\{u_{n}(i),i\geq 1\}$ and rate $\lambda$
if it is ON and remains silent if it is OFF. The lengths of the
ON-periods are identically distributed and so are the lengths of
OFF-periods, and furthermore, both of their distributions can be
heavily tailed with infinite variance. Specifically, for any
distribution $F$, we denote by $\bar{F}=1-F$ the complementary (or
right tail) distribution, and by $F_{1}$ and $F_{2}$ the
distributions for ON- and OFF-periods with probability density
functions $f_{1}$ and $f_{2}$ respectively. Their means and
variances are denoted by $\mu_{i}$ and $\sigma_{i}^{2}$ for $i=1,2$.
In what follows, we assume that as $x\rightarrow\infty$,
\begin{eqnarray}
&&\mbox{either}\;\;\bar{F}_{i}(x)\sim x^{-\alpha_{i}}L_{i}(x)
\;\;\mbox{with}\;\;1<\alpha_{i}<2\;\;\mbox{or}\;\;
\sigma_{i}^{2}<\infty, \elabel{tailI}
\end{eqnarray}
where $L_{i}>0$ is a slowly varying function at infinity, that is,
\begin{eqnarray}
&&\lim_{x\rightarrow\infty}\frac{L_{i}(tx)}{L_{i}(x)}=1\;\;
\mbox{for any}\;\;t>0. \nonumber
\end{eqnarray}
Note that the mean $\mu_{i}$ is always finite but the variance
$\sigma_{i}^{2}$ is infinite when $\alpha_{i}<2$, and furthermore,
one distribution may have finite variance and the other has an
infinite variance since $F_{1}$ and $F_{2}$ are allowed to be
different. The sizes of transmitted packets (service times) form an
i.i.d. random sequence $\{v^{N}(i)=v(i)/\mu^{N},i\geq 1\}$, where
$\mu^{N}$ is the rate of transmission corresponding to each $N$ and
$\{v(i):i\geq 1\}$ is an i.i.d. random sequence with mean 1 and
variance $\sigma_{v}^{2}$, moreover, $\{v(i):i\geq 1\}$ is
independent of the arrival processes.

To derive our queueing dynamical equation, we introduce more
notations. For a single source $n\in\{1,...,N\}$, it follows from
the explanation in Mikosch {\em et al.}~\cite{mikres:nettra} that
the alternating ON/OFF periods can be described by a stationary
binary process $W_{n}=\{W_{n}(t),t\geq 0\}$: $W_{n}(t)=1$ means that
input traffic is in an ON-period at time $t$ and $W_{n}(t)=0$ means
that input traffic is in an OFF-period, and moreover, the mean of
$W_{n}$ is given by
\begin{eqnarray}
&&\gamma=EW_{n}(t)=P(W_{n}(t)=1)=\mu_{1}/(\mu_{1}+\mu_{2}).
\elabel{wmean}
\end{eqnarray}
Let $T_{n}(t)$ denote the cumulative amount of time which
the $n$th source is ON during time interval $[0,t]$, that is,
\begin{eqnarray}
&&T_{n}(t)=\int_{0}^{t}W_{n}(s)ds. \elabel{tntime}
\end{eqnarray}
Let $A_{n}(t)$ be the total number of packets arrived at the
server from the $n$th source during $[0,t]$, namely,
\begin{eqnarray}
&&A_{n}(t)=\sup\left\{m,\sum_{i=1}^{m}u_{n}(i)\leq T_{n}(t)\right\},
\elabel{annumber}
\end{eqnarray}
which exhibits {\em long range dependence} if $\sigma_{1}$ and
$\sigma_{2}$ are not finite simultaneously (see, for instance, Ryu
and Lowen~\cite{ryulow:poipro}). Moreover, let $A^{N}(t)$ be the
total number of packets transmitted to the server by time $t$ summed
over all $N$ sources, that is,
\begin{eqnarray}
&&A^{N}(t)=\sum_{n=1}^{N}A_{n}(t), \elabel{totalnum}
\end{eqnarray}
and let $S^{N}(t)$ be the total number of packets that finished
service at the server if her keep busy in $[0,t]$, that is,
\begin{eqnarray}
&&S^{N}(t)=\sup\left\{m,V^{N}(m)\leq t\right\}, \elabel{snnumber}
\end{eqnarray}
where
\begin{eqnarray}
&&V^{N}(m)=\sum_{i=1}^{m}v^{N}(i). \elabel{vsumn}
\end{eqnarray}
Then the queue length process $Q^{N}(t)$ which is the number
of packets including the one being served at the server at
time $t$ can be represented by
\begin{eqnarray}
&&Q^{N}(t)=A^{N}(t)-S^{N}(B^{N}(t)), \elabel{qb}
\end{eqnarray}
where we assume that the initial queue length is zero for
convenience, $B^{N}(t)$ is the cumulative amount of time that
the server is busy by time $t$. In the following analysis, we
will employ FIFO and {\em non-idling} service discipline under
which the server is never idle when there are packets waiting
to be served. Hence the total busy time can be represented as
\begin{eqnarray}
&&B^{N}(t)=\int_{0}^{t}I\{Q^{N}(s)>0\}ds, \nonumber
\end{eqnarray}
where $I\{\cdot\}$ is the indicator function. Finally, we
introduce the below workload process which measures the delay
of a packet staying in the system,
\begin{eqnarray}
&&L^{N}(t)=V^{N}(A^{N}(t))-B^{N}(t). \elabel{workloadn}
\end{eqnarray}

\section{Heavy traffic limit theorems}\label{limthe}

We are interested in the behaviors of the queueing process
$Q^{N}(\cdot)$ and the workload process $L^{N}(\cdot)$ under
suitable scaling and under the condition that the load of the server
closely approaches the service capacity when the source number $N$
gets large enough. In order to state our main theorems, we introduce
the below notations for convenience, which are adapted from Taqqu
{\em et al.}~\cite{taqwil:profun}. When $1<\alpha_{i}<2$, set
$a_{i}=(\Gamma(2-\alpha_{i}))/(\alpha_{i}-1)$. When
$\sigma_{i}^{2}<\infty$, set $\alpha_{i}=2$, $L_{i}\equiv 1$ and
$a_{i}=\sigma_{i}^{2}/2$. Moreover, let
\begin{eqnarray}
&&b=\lim_{x\rightarrow\infty}t^{\alpha_{2}-\alpha_{1}}
\frac{L_{1}(x)}{L_{2}(x)}. \nonumber
\end{eqnarray}
If $0<b<\infty$ (implying $\alpha_{1}=\alpha_{2}$ and
$b=\lim_{x\rightarrow\infty}L_{1}(x)/L_{2}(x)$), set
$\alpha_{min}=\alpha_{1}$,
\begin{eqnarray}
&&\pi^{2}=\frac{2(\mu_{2}^{2}a_{1}b+\mu_{1}^{2}a_{2})}
{(\mu_{1}+\mu_{2})^{3}\Gamma(4-\alpha_{min})}\;\;\;
\mbox{and}\;\;\;L=L_{2}; \elabel{pisquareI}
\end{eqnarray}
if, on the other hand, $b=0$ or $b=\infty$,
\begin{eqnarray}
&&\pi^{2}=\frac{2\mu_{max}^{2}a_{min}}
{(\mu_{1}+\mu_{2})^{3}\Gamma(4-\alpha_{min})}\;\;\;
\mbox{and}\;\;\;L=L_{min}, \elabel{pisquareII}
\end{eqnarray}
where min is the index 1 if $b=\infty$ (e.g. if
$\alpha_{1}<\alpha_{2}$) and is the index 2 if $b=0$, max denoting
the other index.

\subsection{Reflecting Gaussian process as the limit}

\begin{condition}\label{conditionI}({\bf heavy traffic condition})
For each $N$, let the service rate $\mu^{N}$ be given by
\begin{eqnarray}
&&\mu^{N}=N\lambda\gamma+\theta\sqrt N, \elabel{ofI}
\end{eqnarray}
where $\theta$ is some positive constant.
\end{condition}
In addition, we need the below conditions on the distributions of
$F_{1}$ and $F_{2}$:
\begin{eqnarray}
&&F_{i}(x)\;\;(i=1,2)\;\;
\mbox{is absolutely continuous in terms of}\;\; x;
\elabel{IfI}\\
&&\mbox{The density}\;\;f_{i}(x)\;\;(i=1,2)\;\;\mbox{of}\;\;
F_{i}\;\;\mbox{satisfies}
\lim_{x\rightarrow 0^{+}}f_{i}(x)<\infty.
\elabel{IIIfI}
\end{eqnarray}
Before we state our main theorems, we define the scaling
processes for each $N$ as follows,
\begin{eqnarray}
&&\tilde{Q}^{N}(\cdot)\equiv\frac{1}{\sqrt N}Q^{N}(\cdot),\;\;\;
\tilde{L}^{N}(\cdot)\equiv\frac{\mu^{N}}{\sqrt N}L^{N}(\cdot).
\elabel{qnln}
\end{eqnarray}
\begin{theorem}\label{heavytra}
Under condtions \eq{ofI}-\eq{IIIfI} and as $N\rightarrow\infty$,
both $\tilde{Q}^{N}(\cdot)$ and $\tilde{L}^{N}(\cdot)$ converge in
distribution under Skorohod topology to a reflecting Gaussian
process $\tilde{Q}(\cdot)$ given by
\begin{eqnarray}
&&\tilde{Q}(\cdot)=\tilde{A}(\gamma\cdot)+\lambda\tilde{T}(\cdot)
-\tilde{S}(\lambda\gamma\cdot)-\theta\cdot+\tilde{I}(\cdot) \geq 0,
\elabel{qalsy}
\end{eqnarray}
where the three processes $\tilde{A}(\gamma\cdot)$,
$\tilde{S}(\lambda\gamma\cdot)$ and $\tilde{T}(\cdot)$ are
independent each other, and furthermore, $\tilde{A}(\gamma\cdot)$ is
a Brownian motion with mean zero and variance function
$\lambda\gamma\cdot$, $\tilde{S}(\lambda\gamma\cdot)$ is also a
Brownian motion with mean zero and variance function
$\lambda\gamma\sigma_{v}^{2}\cdot$ , $\tilde{T}(\cdot)$ is a
Gaussian process with a.s. continuous sample paths, mean zero and
stationary increments, whose covariance and variance functions
satisfy
\begin{eqnarray}
Cov(\tilde{T}(t),\tilde{T}(s))&=&\frac{1}{2}\left(Var(\tilde{T}(t))
+Var(\tilde{T}(s))-Var(\tilde{T}(t-s))\right),
\elabel{vcovariance}\\
Var(\tilde{T}(t))&\sim&\left\{\begin{array}{ll}
       \pi^{2}t^{2H}L(t)&\mbox{as}\;\;t\rightarrow\infty\;\;
       \mbox{for}\;\;1<\alpha_{min}<2,\\
       \pi^{2}t&\mbox{as}\;\;t\rightarrow\infty\;\;\mbox{and}
       \;\;\alpha_{min}=2,
       \end{array}
\right. \elabel{vvariance}
\end{eqnarray}
where $H$ is the Hurst parameter given by $H=(3-\alpha_{min})/2$.
Moreover, $\tilde{I}(\cdot)$ in \eq{qalsy} is a non-decreasing
process with $\tilde{I}(0)=0$ and satisfies
\begin{eqnarray}
&&\int_{0}^{\infty}\tilde{Q}(s)d\tilde{I}(s)=0. \nonumber
\end{eqnarray}
\end{theorem}
\begin{remark}
More discussions about reflected Gaussian processes, readers are
referred to Whitt~\cite{whi:stopro}. From the theorem, we have the
following observations. When $1<\alpha_{min}<2$, we have that
$1/2<H<1$ which implies that the process $\tilde{T}(\cdot)$ exhibits
long range dependence. When $\alpha_{i}=2$ for $i=1,2$, the ON- and
OFF-periods both have finite variance and hence we have that $H=1/2$
and $L=1$, which imply that $\tilde{T}(\cdot)$ exhibits short range
dependence. Finally, the results given in the theorem can be
considered as a kind of extension of Little's formula for M/M/1
queueing model or considered as satisfying certain state space
collapse property.
\end{remark}

\subsection{Reflecting fractional Brownian motion as the limit}

In this subsection, we suppose that at least one of $\sigma_{i}^{2}$
$(i=1,2)$ is infinite. To further discussion, we need to introduce
another time-scaling parameter $R$ and assume that $N=N(R)$ goes to
infinite as $R\rightarrow\infty$. Moreover, we assume that $N$ is
taken to satisfy the below fast growth condition (and see more
discussion in Mikosch {\em et al.}~\cite{mikres:nettra})
\begin{eqnarray}
&&NR\bar{F}_{L}(R)\rightarrow\infty,\;\;\;\mbox{as}\;\;\;
R\rightarrow\infty, \elabel{fgc}
\end{eqnarray}
where $\bar{F}_{L}=\bar{F}_{i}$ if $L=L_{i}$ and $L$ is defined
in \eq{pisquareI} and \eq{pisquareII}. Notice that $\eq{fgc}$
implies $NR^{1-\alpha_{min}}L(R)\rightarrow\infty$.

\begin{condition}\label{condition}({\bf heavy traffic condition})
For each $N$ and $R$, let the service rate $\mu^{R}$ be given
by
\begin{eqnarray}
&&\mu^{R}=N\lambda\gamma
+\theta\left(NR^{1-\alpha_{min}}L(R)\right)^{1/2}, \elabel{heavconI}
\end{eqnarray}
where $\theta$ is some positive constant.
\end{condition}
Next, let $d_{R}$ be the normalization sequence given by
\begin{eqnarray}
&&d_{R}=(NR^{3-\alpha_{min}}L(R))^{1/2}, \elabel{definedt}
\end{eqnarray}
and define
\begin{eqnarray}
&&\tilde{Q}^{R}(\cdot)\equiv\frac{1}{d_{R}}Q^{N}(R\cdot),\;\;
\tilde{L}^{R}(\cdot)\equiv\frac{\mu^R}{d_{R}}L^{N}(R\cdot).
\elabel{tildeqr}
\end{eqnarray}
\begin{theorem}\label{heavytraI}
Assuming that conditions~\eq{fgc} and~\eq{heavconI} hold,
then as $R\rightarrow\infty$, both $\tilde{Q}^{R}(\cdot)$
and $\tilde{L}^{R}(\cdot)$ converge in distribution under
Skorohod topology to a process $\tilde{Q}_{H}(\cdot)$
given by
\begin{eqnarray}
\tilde{Q}_{H}(\cdot)=\lambda\pi B_{H}(\cdot)-
\theta\cdot+\tilde{I}_{H}(\cdot)\geq 0, \elabel{fqalsy}
\end{eqnarray}
where $B_{H}(\cdot)$ is a standard FBM, and $\tilde{I}_{H}(\cdot)$
is a non-decreasing process with $\tilde{I}(0)=0$ and satisfies
\begin{eqnarray}
&&\int_{0}^{\infty}\tilde{Q}_{H}(s)d\tilde{I}_{H}(s)=0. \nonumber
\end{eqnarray}
\end{theorem}
\begin{remark}Standard FBM is a mean zero Gaussian process with
a.s. continuous sample paths and whose covariance structure is
as follows
\begin{eqnarray}
&&Cov(B_{H}(t),B_{H}(s))=\frac{1}{2}\left(|t|^{2H}+|s|^{2H}-
|t-s|^{2H}\right). \nonumber
\end{eqnarray}
\end{remark}

\section{Proof of main theorems}\label{proofI}

Let $T^{N}(t)$ be the total cumulative amount of ON time summed over
all $N$ sources, that is,
\begin{eqnarray}
&&T^{N}(t)\equiv\sum_{n=1}^{N}T_{n}(t)=\int_{0}^{t}W^{N}(s)ds,
\elabel{tNt}
\end{eqnarray}
where $W^{N}(\cdot)$ is the superposition of $W_{n}(\cdot)$ for
$n=1,...,N$, that is, for each $t\in [0,\infty)$,
\begin{eqnarray}
&&W^{N}(t)=\sum_{n=1}^{N}W_{n}(t). \elabel{Wnt}
\end{eqnarray}
Moreover, let $\bar{A}(t)$ denote the cumulative number of arrival
packets to the server during the time interval $[0,t]$, that is,
\begin{eqnarray}
&&\bar{A}(t)= \sup\left\{m:\;\sum_{i=1}^{m}u(i)\leq t\right\},
\elabel{barA}
\end{eqnarray}
where $\{u(i),i=1,2,...\}$ is an exponentially distributed
random sequence with mean value $1/\lambda$, which is independent
of all processes mentioned before. Then we have the below lemma.
\begin{lemma}\label{equivalent}
The stochastic processes $A^{N}(\cdot)$ in \eq{totalnum} and
$\bar{A}(T^{N}(\cdot))$ in \eq{barA} have the same
distribution.
\end{lemma}
\begin{proof}
To show that $A^{N}(\cdot)$ and $\bar{A}(T^{N}(\cdot))$ have the
same distribution, it suffices to show that they have the same
finite-dimensional distribution for an arbitrary positive integer
$k$ and arbitrary numbers $t_{1},...,t_{k}\in [0,\infty)$ according
to Proposition 2.2 in Kallenberg~\cite{kal:foumod}.

Notice that the process $W^{N}(\cdot)$ in \eq{Wnt} takes values
in the set ${\cal N}=\{0,1,...,N\}$ and has the piecewise constant
sample paths given by
\begin{eqnarray}
&&x(t)=\sum_{i=1}^{M}n_{i-1}I\{s_{i-1}\leq t<s_{i}\}, n_{i-1}\in
{\cal N},n_{i-1}\neq n_{i}, \elabel{spath}
\end{eqnarray}
where $s_{0},s_{1},...,s_{M}$ with $s_{0}=0$ and $s_{M}=\infty$ is a
partition of the interval $[0,\infty)$ and $M$ is a positive integer
or infinite. Then we use $D_{{\cal N}}[0,\infty)$ to denote the set
of all of these functions defined in \eq{spath}. Obviously, it is a
subset of the Skorohod topological space $D_{E}[0,\infty)$. Under
the same topology, $D_{{\cal N}}[0,\infty)$ becomes a measurable
space in its own right when endowed with the Borel $\sigma$-field
$A\cap {\cal B}=\{A\cap B,B\in {\cal B}\}$ where ${\cal B}$ is the
Borel $\sigma$-field in $D_{E}[0,\infty)$ (see, for example,
Kallenberg~\cite{kal:foumod}). Then there is a probability
distribution $F^{N}(\cdot)$ on $D_{{\cal N}}[0,\infty)$ for the
process $W^{N}(\cdot)$ in \eq{Wnt}, which is uniquely determined by
the length distributions of ON- and OFF-periods and the source
number $N$ (here, for our purpose, we will not derive the explicit
expression of $F^{N}(\cdot)$).

Basing on the above observation, we first consider the
one-dimensional case. For each $t\geq 0$ and each nonnegative
number $m$, it follows from the independent and stationary
increment properties of Poisson process that
\begin{eqnarray}
&&P\{A^{N}(t)=m\}=\int_{D_{{\cal N}}[0,\infty)}
\left\{\sum_{i=1}^{c}N_{n_{i}\lambda}(\Delta s_{i})=m
|W^{N}(\cdot)=x(\cdot)\right\}F^{N}(dx), \nonumber
\end{eqnarray}
where $x(\cdot)$ is a sample path as defined in \eq{spath},
$N_{n_{i}\lambda}(\Delta s_{i})$ is the number of arrival packets
for the Poisson process with arrival rate $n_{i}\lambda$ during the
time interval $\Delta s_{i}=\min\{s_{i},t\}-\min\{s_{i-1},t\}$ for
$i\in\{1,...,M\}$, and the integer $c$ is given by
$c=1+\sup\{i:s_{i}<t\}$. Then by the independent and stationary
increment properties again, we have,
\begin{eqnarray}
P\{A^{N}(t)=m\}
&=&\int_{D_{{\cal N}}[0,\infty)}
P\left\{\sum_{i=0}^{N}N_{i\lambda}(\Delta\bar{s}_{i})=m
|W^{N}(\cdot)=x(\cdot)\right\}F^{N}(dx)
\nonumber\\
&=&\int_{D_{{\cal N}}[0,\infty)}
P\left\{\sum_{i=1}^{N}N_{\lambda}(i\Delta\bar{s}_{i})=m
|W^{N}(\cdot)=x(\cdot)\right\}F^{N}(dx)
\nonumber\\
&=&\int_{D_{{\cal N}}[0,\infty)}
P\left\{N_{\lambda}(\tau^{N}(t))=m
|W^{N}(\cdot)=x(\cdot)\right\}F^{N}(dx)
\nonumber\\
&=&P\{\bar{A}(T^{N}(t))=m\}, \nonumber
\end{eqnarray}
where $\Delta\bar{s}_{i}$ is the summation of time intervals
during which the arrival rate for the associated Poisson process
is $i\lambda$, and $\tau^{N}(t)$ is the total cumulative amount
of ON time from all $N$ sources up to time $t$ along the sample
path $x(\cdot)$.

Secondly, we consider the two-dimensional case (we will omit
the discussion for more higher-dimensional cases since they are
similar). For any $t_{1},t_{2}\in [0,\infty)$ with $t_{1}<t_{2}$,
and nonnegative integers $m_{1}$ and $m_{2}$, it follows from
the independent and stationary increment properties and the
definition of conditional probability that
\begin{eqnarray}
&&P\left\{A^{N}(t_{1})=m_{1},A^{N}(t_{2})=m_{2}\right\}
\nonumber\\
&=&
\int_{D_{{\cal N}}[0,\infty)}
P\left\{A^{N}(t_{1})=m_{1}|W^{N}(\cdot)=x(\cdot)\right\}
\nonumber\\
&&\;\;\;\;\;\;\;\;\;\;\;\;\;\;
P\left\{A^{N}(t_{2}-t_{1})=m_{2}-m_{1}|
W^{N}(\cdot)=x(\cdot)\right\}
F^{N}(dx)
\nonumber\\
&=&
\int_{D_{{\cal N}}[0,\infty)}
P\left\{N_{\lambda}(\tau(t_{1}))=m_{1}|W^{N}(\cdot)=x(\cdot)\right\}
\nonumber\\
&&\;\;\;\;\;\;\;\;\;\;\;\;\;\;
P\left\{N_{\lambda}(\tau(t_{2}-t_{1}))=m_{2}-m_{1}|
W^{N}(\cdot)=x(\cdot)\right\}
F^{N}(dx)
\nonumber\\
&=&
P\left\{\bar{A}(T^{N}(t_{1}))=m_{1},
\bar{A}(T^{N}(t_{2}))=m_{2}\right\}.
\nonumber
\end{eqnarray}
where $\tau(t_{2}-t_{1})$ is the total cumulative amount of ON time
from all $N$ sources during time interval $[t_{1},t_{2})$ along the
path $x(\cdot)$. Hence we have proved that $A^{N}(\cdot)$ and
$\bar{A}(T^{N}(\cdot))$ have the same distribution.
$\Box$
\end{proof}

\subsection{Proof of Theorem~\ref{heavytra}}

First of all, we define some scaled and centered processes. For each
$t\geq 0$ and $N\geq 1$, let
\begin{eqnarray}
&&
\tilde{\bar{A}}^{N}(t)\equiv
\frac{1}{\sqrt N}\left(\bar{A}(Nt)-\lambda N t\right),
\elabel{difsvI}\\
&&
\tilde{S}^{N}(t)\equiv
\frac{1}{\sqrt N}\left(S^{N}(t)-\mu^{N}t\right),
\elabel{difsvIII}\\
&&
\tilde{T}^{N}(t)\equiv
\frac{1}{\sqrt N}\left(T^{N}(t)-\gamma Nt\right)
=
\int_{0}^{t}\frac{1}{\sqrt N}
\sum_{n=1}^{N}\left(W_{n}(s)-\gamma N\right)ds.
\elabel{tnttilde}
\end{eqnarray}
Then we have the following lemma.
\begin{lemma}\label{threetilde}
There exist three independent processes $\tilde{A}(\cdot)$,
$\tilde{S}(\lambda\gamma\cdot)$ and $\tilde{T}(\cdot)$ such that
\begin{eqnarray}
&&(\tilde{\bar{A}}^{N}(\cdot),\tilde{S}^{N}(\cdot),
\tilde{T}^{N}(\cdot))\Rightarrow
(\tilde{A}(\cdot),\tilde{S}(\lambda\gamma\cdot),\tilde{T}(\cdot))
\;\;\;\mbox{as}\;\;\;N\rightarrow\infty,
\end{eqnarray}
where $\tilde{A}(\cdot)$ is a Brownian motion with mean 0 and
variance function $\lambda\cdot$, $\tilde{S}(\lambda\gamma\cdot)$ is
a Brownian motion with mean zero and variance function
$\lambda\gamma\sigma_{v}^{2}\cdot$, $\tilde{T}(\cdot)$ is a Gaussian
process with stationary increments, mean 0, stationary increments,
whose covariance and variance functions are as given in
\eq{vcovariance}-\eq{vvariance}.
\end{lemma}
\begin{proof}
First of all, it follows from Functional Central Limit Theorem
(e.g., Chen and Yao~\cite{cheyao:funque}) that
\begin{eqnarray}
&&\tilde{\bar{A}}^{N}(\cdot)\equiv \frac{1}{\sqrt
N}\left(\bar{A}(N\cdot) -\lambda N\cdot\right) \Rightarrow
\tilde{A}(\cdot),
\end{eqnarray}
where $\tilde{A}(\cdot)$ is a Brownian motion with mean zero
and variance $\lambda\cdot$.

Secondly, for each $t\geq 0$, we have,
\begin{eqnarray}
\tilde{S}^{N}(t)
&=&
\frac{1}{\sqrt N}\left(\mbox{sup}\left\{k:
v^{N}(1)+...+v^{N}(k)\leq t\right\}-\mu^{N}t\right)
\nonumber\\
&=&
\frac{1}{\sqrt N}\left(\mbox{sup}\left\{k:
v(1)+...+v(k)\leq \mu^{N}t\right\}-\mu^{N}t\right)
\nonumber\\
&=&
\frac{1}{\sqrt N}\left(S_{1}^{N}(N\mu_{1}^{N}t)
-N\mu_{1}^{N}t\right),
\nonumber
\end{eqnarray}
where in the last equation, $\mu_{1}^{N}$ is given by
\begin{eqnarray}
&&\mu_{1}^{N}=\lambda\gamma+\frac{\theta}{\sqrt N}, \nonumber
\end{eqnarray}
and $S_{1}^{N}(\cdot)$ is the counting process corresponding to the
i.i.d. normalized random sequence $\{v(i),i\geq 1\}$ with mean 1 and
variance $\sigma_{v}^{2}$. It is obvious that
$\mu^{N}_{1}\rightarrow\lambda\gamma$ as $N\rightarrow\infty$. Then
by Functional Central Limit Theorem (e.g., Chen and
Yao~\cite{cheyao:funque}), we have
\begin{eqnarray}
&&\tilde{S}^{N}(\cdot)\Rightarrow\tilde{S}(\lambda\gamma\cdot)
\;\;\;\mbox{as}\;\;\;N\rightarrow\infty, \nonumber
\end{eqnarray}
where $\tilde{S}(\lambda\gamma\cdot)$ is a Brownian motion
with mean zero and variance function
$\lambda\gamma\sigma_{v}^{2}\cdot$.

Thirdly, it follows from conditions \eq{IfI}-\eq{IIIfI} and
Corollary 3.1 in Debicki and Palmowski~\cite{debpal:flumod} that the
below convergence in distribution is true
\begin{eqnarray}
\tilde{W}^{N}(\cdot)\equiv\frac{1}{\sqrt N}
\sum_{n=1}^{N}\left(W_{n}(\cdot)-\gamma N\right)
\Rightarrow\tilde{W}(\cdot), \elabel{wnweakc}
\end{eqnarray}
where $\tilde{W}(\cdot)$ is a stationary centered Gaussian process
with a.s. continuous sample paths (by Hahn~\cite{hah:cenlim} since
$W_{n}(t)$ is stochastically continuous) and covariance function
$\eta(\cdot)$ which satisfies (see the proof of Theorem 1 in Taqque
{\em et al.}~\cite{taqwil:profun} for details),
\begin{eqnarray}
Var\left(\int_{0}^{t}\tilde{W}(u)du\right)=
2\int_{0}^{t}\int_{0}^{v}\eta(u)dudv,\elabel{doubint}
\end{eqnarray}
which has the expression as in \eq{vvariance}. By Skorohod
representation theorem (see, for example, Ethier and
Kurtz~\cite{ethkur:marpro}), we can assume that the convergence in
\eq{wnweakc} is u.o.c. Then we have
\begin{eqnarray}
\tilde{T}^{N}(\cdot)=
\int_{0}^{\cdot}\tilde{W}^{N}(s)ds
\rightarrow
\int_{0}^{\cdot}\tilde{W}(s)ds
\equiv
\tilde{T}(\cdot)
\;\;\;\mbox{u.o.c. as}\;\;\;N\rightarrow\infty.
\nonumber
\end{eqnarray}
Thus by the definition of weak convergence on $C[0,\infty)$ (see,
for example, Whitt~\cite{whi:stopro}), Skorohod representation
theorem and Proposition 14.6 in Kallenberg~\cite{kal:foumod}, the
above u.o.c. convergence implies weak convergence. Now, we show that
$\tilde{T}(\cdot)$ is a Gaussian process. Due to \eq{doubint} and
Theorem 7 in page 128 of~\cite{fud:prothe}, $\tilde{W}(\cdot)$ is
mean square integrable in any  given finite interval $[0,T]$, and
therefore it follows from Theorem 3 in page 142 of~\cite{fud:prothe}
that $\tilde{T}(\cdot)$ is a Gaussian process in $[0,T]$. Since for
any given $n\in\{1,2,...,\}$ and any given
$t_{1},...,t_{n}\in[0,\infty)$, we can find an $T_{1}<\infty$ such
that $t_{1},...,t_{n}$ belong to the common interval $[0,T_{1}]$.
Hence the joint distribution of
$\tilde{T}(t_{1}),...,\tilde{T}(t_{n})$ is normal. Thus we can
conclude that $\tilde{T}(\cdot)$ is a Gaussian process in
$[0,\infty)$, whose variance function is as shown in \eq{doubint}.
Since $\tilde{W}(\cdot)$ is stationary, $\tilde{T}(\cdot)$ has
stationary increments and its covariance function is given by the
expression in \eq{vcovariance} due to Proposition 1(b) in Choe and
Shroff~\cite{choshr:supdis}.

Finally, by the independence assumptions and definitions of
related processes, we know that the three processes
$\tilde{\bar{A}}^{N}(\cdot)$, $\tilde{S}^{N}(\cdot)$ and
$\tilde{T}^{N}(\cdot)$ are independent each other for each $N$.
Thus we can conclude that $\tilde{A}(\cdot)$,
$\tilde{S}(\lambda\gamma\cdot)$ and $\tilde{T}(\cdot)$ are
independent each other. Hence we finish the proof of the lemma.
$\Box$
\end{proof}

To complete the proof of the theorem, for each $t\geq 0$, we
rewrite \eq{qb} as the summation of centered processes and
regulated non-decreasing process as follows,
\begin{eqnarray}
Q^{N}(t)&=&X^{N}(t)+I^{N}(t), \elabel{ncenteredQ}
\end{eqnarray}
where
\begin{eqnarray}
X^{N}(t)&=&
\left(A^{N}(t)-\lambda\gamma Nt\right)
-\left(S^{N}\left(B^{N}(t)\right)-
\mu^{N}B^{N}(t)\right)-\sqrt N\theta t,
\nonumber\\
I^{N}(t)&=&\mu^{N}\int_{0}^{t}I\{Q^{N}(s)=0\}ds.
\nonumber
\end{eqnarray}
The process $I^{N}(\cdot)$ is non-decreasing process and can
increase only when the queue length process $Q^{N}(\cdot)$
reaches zero due to the non-idling service discipline and the
fact that $Q^{N}(t)\geq 0$ for all $t\geq 0$.
\begin{lemma}\label{xnlimit}
\begin{eqnarray}
&&\tilde{X}^{N}(\cdot)\equiv\frac{1}{\sqrt N} X^{N}(\cdot)
\Rightarrow\tilde{X}(\cdot)
=\tilde{A}(\gamma\cdot)+\lambda\tilde{T}(\cdot)
-\tilde{S}(\lambda\gamma\cdot)-\theta\cdot
\;\;\;\mbox{as}\;\;N\rightarrow\infty, \nonumber
\end{eqnarray}
where $\tilde{A}(\gamma\cdot)$ is a Brownian motion with mean 0
and variance function $\lambda\gamma\cdot$.
\end{lemma}
\begin{proof}
First of all, we prove the following claim to be true
\begin{eqnarray}
&&\tilde{E}^{N}(\cdot)\equiv \frac{1}{\sqrt
N}\left(A^{N}(\cdot)-\lambda\gamma N\cdot\right)
\Rightarrow\tilde{A}(\gamma\cdot)+\lambda\tilde{T}(\cdot)
\;\;\mbox{as}\;\;N\rightarrow\infty. \nonumber
\end{eqnarray}
In fact, by Lemma~\ref{equivalent}, it suffices to prove the
following claim,
\begin{eqnarray}
&&\frac{1}{\sqrt N}\left(\bar{A}(T^{N}(\cdot))-\lambda\gamma
N\cdot\right)\Rightarrow\tilde{A}(\gamma\cdot)
+\lambda\tilde{T}(\cdot)\;\;\mbox{as}\;\;N\rightarrow\infty,
\nonumber
\end{eqnarray}
and it is a direct conclusion of Lemma~\ref{threetilde} and
Corollary 13.3.2 of Whitt~\cite{whi:stopro}. Thus, by
Lemma~\ref{threetilde} and the independence assumption, we have the
below joint weak convergence
\begin{eqnarray}
&&\left(\tilde{E}^{N}(\cdot),\tilde{S}^{N}(\cdot)\right) \Rightarrow
\left(\tilde{A}(\gamma\cdot)+\lambda\tilde{T}(\cdot),
\tilde{S}(\lambda\gamma\cdot)\right). \elabel{jwkcon}
\end{eqnarray}
Moreover, by Skorohod representation theorem, we can assume that the
above convergence is u.o.c. a.s. Thus it follows from \eq{jwkcon}
that
\begin{eqnarray}
&&\left(\frac{1}{N}A^{N}(\cdot),\frac{1}{N}{S}^{N}(\cdot)\right)
\rightarrow(\lambda\gamma\cdot,\lambda\gamma\cdot)\;\;
\mbox{u.o.c.}\;\;\mbox{a.s.}
\elabel{fluidlimc}
\end{eqnarray}
Then, due to \eq{fluidlimc}, the conditions stated in Theorem 6.5
of~\cite{cheyao:funque} are satisfied. So, by the same theorem
of~\cite{cheyao:funque}, we know that, for each $t\geq 0$ and as
$N\rightarrow\infty$,
\begin{eqnarray}
&&\max_{0\leq s\leq t}|B^{N}(s)-s|\rightarrow 0. \elabel{btzo}
\end{eqnarray}
Therefore, by the above discussions and the fact that the associated
limiting processes have a.s. continuous sample paths, we have
\begin{eqnarray}
\left\|\tilde{X}^{N}(\cdot)-\tilde{X}(\cdot)\right\|_{t}
&\leq&
\left\|\tilde{E}^{N}(\cdot)-\tilde{A}(\gamma\cdot)
+\lambda\tilde{T}(\cdot)\right\|_{t}+
\left\|\tilde{S}^{N}(B^{N}(\cdot))
-\tilde{S}(\lambda\gamma B^{N}(\cdot))\right\|_{t}
\nonumber\\
&&+\left\|\tilde{S}(B^{N}(\cdot))
-\tilde{S}(\lambda\gamma\cdot)\right\|_{t}
\nonumber\\
&\leq&
\left\|\tilde{E}^{N}(\cdot)-\tilde{A}(\gamma\cdot)
+\lambda\tilde{T}(\cdot)\right\|_{t}+
\left\|\tilde{S}^{N}(\cdot)
-\tilde{S}(\lambda\gamma\cdot)\right\|_{t}
\nonumber\\
&&+\left\|\tilde{S}(B^{N}(\cdot))
-\tilde{S}(\lambda\gamma\cdot)\right\|_{t}
\nonumber\\
&\rightarrow& 0\;\;\;\mbox{a.s. as}\;\;N\rightarrow\infty, \nonumber
\end{eqnarray}
where in the second inequality, we used the fact that
$B^{N}(t)\leq t$ for each $t\geq 0$ and in the last claim, we
also used the fact that $\tilde{S}(\cdot)$ is continuous. Thus
\begin{eqnarray}
&&\tilde{X}^{N}(\cdot)\rightarrow\tilde{X}(\cdot)\;\;\;\mbox{u.o.c.
a.s. as}\;\;N\rightarrow\infty. \nonumber
\end{eqnarray}
Hence by Proposition 5.3 in Chapter 3 of Ethier and
Kurtz~\cite{ethkur:marpro}, the lemma is proved. $\Box$
\end{proof}

Next, similar to the discussion as in \eq{ncenteredQ}, let
\begin{eqnarray}
&&\tilde{V}^{N}(t)=\frac{1}{\sqrt N}\left(\mu^{N}V^{N}(t)-t\right),
\elabel{vnt}
\end{eqnarray}
and rewrite \eq{workloadn} as the summation of centered processes
and regulated non-decreasing process as follows,
\begin{eqnarray}
&&\mu^{N}L^{N}(t)=Z^{N}(t)+I^{N}(t), \elabel{centerLN}
\end{eqnarray}
where
\begin{eqnarray}
&&Z^{N}(t)=\left(\mu^{N}V^{N}(A^{N}(t))-A^{N}(t)\right)
+\left(A^{N}(t)-N\lambda\gamma t\right)-{\sqrt N}\theta t. \nonumber
\end{eqnarray}
Then we have the following lemma.
\begin{lemma}\label{znttilde}
\begin{eqnarray}
&&(\tilde{V}^{N}(\cdot),\tilde{\bar{A}}^{N}(\cdot),
\tilde{T}^{N}(\cdot))\Rightarrow
(\tilde{V}(\cdot),\tilde{A}(\lambda\gamma\cdot),\tilde{T}(\cdot))
\;\;\;\mbox{as}\;\;\;N\rightarrow\infty, \elabel{vncon}
\end{eqnarray}
where $\tilde{V}(\cdot)$, $\tilde{A}(\cdot)$ and
$\tilde{T}(\cdot)$ are independent Brownian motions, moreover,
$\tilde{V}(\cdot)$ is of mean zero and variance function
$\sigma_{v}^{2}\cdot$, $\tilde{A}(\cdot)$ and
$\tilde{T}(\cdot)$ are given as before. Moreover,
\begin{eqnarray}
&&\tilde{Z}^{N}(\cdot)\equiv\frac{1}{\sqrt N} Z^{N}(\cdot)
\Rightarrow\tilde{X}(\cdot)
=\tilde{A}(\gamma\cdot)+\lambda\tilde{T}(\cdot)
-\tilde{S}(\lambda\gamma\cdot)-\theta\cdot
\;\;\;\mbox{as}\;\;N\rightarrow\infty. \elabel{zvncon}
\end{eqnarray}
\end{lemma}
\begin{proof}
By applying Functional Central Limit Theorem and the same
explanation as in Lemma~\ref{threetilde}, one can prove the
convergence stated in \eq{vncon}. Then it follows from
\eq{vncon}, Lemma~\ref{equivalent} and random time change
theorem that
\begin{eqnarray}
&&\tilde{Z}^{N}(\cdot)\equiv\frac{1}{\sqrt N} Z^{N}(\cdot)
\Rightarrow\tilde{V}(\lambda\gamma\cdot)+\tilde{A}(\gamma\cdot)
+\lambda\tilde{T}(\cdot)-\theta\cdot
\;\;\;\mbox{as}\;\;N\rightarrow\infty. \nonumber
\end{eqnarray}
Notice that $\tilde{V}(\lambda\gamma\cdot)$ and
$-\tilde{S}(\lambda\gamma\cdot)$ have the same distribution,
we can conclude that the claim stated in \eq{zvncon} is true.
$\Box$
\end{proof}

\vskip 0.15in \noindent {\bf Proof of Theorem~\ref{heavytra}} \vskip
0.10in Once the above lemmas are obtained, we can go over the
following standard procedure to finish the proof of the theorem. By
Skorohod representation theorem, we suppose that the convergence in
Lemma~\ref{xnlimit} is u.o.c. Then, by \eq{ncenteredQ} and according
to Theorem 6.1 in Chen and Yao~\cite{cheyao:funque}, there uniquely
exist a pair of regulated mappings $\phi$ and $\psi$, which are
continuous, such that for each $t\geq 0$,
\begin{eqnarray}
\tilde{I}^{N}(t)&=&\frac{1}{\sqrt N}I^{N}(t)
=\phi\left(\tilde{X}^{N}(t)\right) =\sup_{0\leq s\leq
t}\left(\tilde{X}^{N}(s)\right)^{-},
\nonumber\\
\tilde{Q}^{N}(t)&=& \frac{1}{\sqrt
N}Q^{N}(t)=\psi\left(\tilde{X}^{N}(t)\right)
=\tilde{X}^{N}(t)+\phi\left(\tilde{X}^{N}(t)\right)\geq 0, \nonumber
\end{eqnarray}
where $x^{-}(s)=\max\{-x(s),0\}$. Then by continuous
mapping theorem and Lemma~\ref{xnlimit}, we have, as
$N\rightarrow\infty$,
\begin{eqnarray}
I^{N}(\cdot)&\rightarrow&\tilde{I}(\cdot)\equiv
\phi\left(\tilde{X}(\cdot)\right)
\;\;\;\;\mbox{a.s.}\;\;\;\mbox{u.o.c.},
\elabel{nbarIlimit}\\
\tilde{Q}^{N}(\cdot)&\rightarrow&\tilde{Q}(\cdot)
\equiv\psi\left(\tilde{X}(\cdot)\right)\geq 0
\;\;\;\;\mbox{a.s.}\;\;\;\mbox{u.o.c.}
\elabel{nbarqlimit}
\end{eqnarray}
Obviously, $\tilde{I}(\cdot)$ and $\tilde{Q}(\cdot)$
have a.s. continuous sample paths, and moreover,
$\tilde{I}(\cdot)$ is non-decreasing with
$\tilde{I}(0)=0$. Since $\tilde{Q}^{N}(t)\geq 0$ and
$I^{N}(t)$ increases only at times $t$ such that
$\tilde{Q}^{N}(t)=0$, we have for each $T>0$,
\begin{eqnarray}
&&\int_{0}^{T}\tilde{Q}^{N}(t)\wedge 1d\tilde{I}^{N}(t)=0.
\elabel{nbarintlimit}
\end{eqnarray}
Define
\begin{eqnarray}
&&f:x\in R\rightarrow f(x)=x\wedge 1. \nonumber
\end{eqnarray}
Clearly, we have $f\in C_{b}(R)$. Then by \eq{nbarIlimit},
\eq{nbarqlimit}, \eq{nbarintlimit} and Lemma 8.3 in Dai and
Dai~\cite{daidai:heatra}, we have
\begin{eqnarray}
&&\int_{0}^{T}\tilde{Q}(t)\wedge 1d\tilde{I}(t)=0 \;\;\;\mbox{for
all}\;\;\;T>0. \nonumber
\end{eqnarray}
Hence $\tilde{I}(\cdot)$ increases only at times $t$
such that $\tilde{Q}(t)=0$.

Finally, by Lemma~\ref{znttilde} and the same procedure
as above, one can prove the weak convergence for the processes
of $\tilde{L}^{N}(\cdot)$ as $N\rightarrow\infty$. $\Box$\\

\subsection{Proof of Theorem~\ref{heavytraI}}

\begin{lemma}\label{uvwlimit}
Let $\beta=1-\alpha_{min}/2$. Then, as $T\rightarrow\infty$, we
have,
\begin{eqnarray}
&&U(T)\equiv T^{\beta}L(T)^{1/2}\rightarrow\infty,
\elabel{utlimit}\\
&&V(T)\equiv T^{\alpha_{min}/2-1/2}/L(T)^{1/2}\rightarrow\infty.
\elabel{wtlimit}
\end{eqnarray}
\end{lemma}
\begin{proof}
Since $L(T)$ is a slowly varying function and $1<\alpha_{min}<2$,
we know that $U(T)$ is a regularly varying function with
index $0<\beta<1/2$, that is, for $x>0$,
\begin{eqnarray}
\lim_{T\rightarrow\infty}\frac{U(Tx)}{U(T)}=x^{\beta}.
\nonumber
\end{eqnarray}
Then, take $0<\epsilon<\beta$, it follows from Proposition 0.8 in
Resnick~\cite{res:extval} that there is a fixed $T_{0}$ such that
for $x\geq 1$ and $T\geq T_{0}$, we have
\begin{eqnarray}
&&U(Tx)>(1-\epsilon)x^{\beta-\epsilon}U(T). \nonumber
\end{eqnarray}
Let $x\rightarrow\infty$ in the above inequality, we know
that \eq{utlimit} is true.

Similarly, $V(T)$ is a regularly varying function with index
$0<\alpha_{min}/2-1/2<1/2$, then by the same reason as above,
we know that \eq{wtlimit} holds.
$\Box$
\end{proof}

Now for each $t\geq 0$, we rewrite \eq{tildeqr} as the summation of
centered processes and regulated non-decreasing process as follows,
\begin{eqnarray}
\tilde{Q}^{R}(t)&=&X^{R}(t)+I^{R}(t), \elabel{centeredQ}
\end{eqnarray}
where
\begin{eqnarray}
X^{R}(t)&=&\frac{1}{d_{R}}
\left(A^{N}(Rt)-\lambda\gamma NRt\right)
-\frac{1}{d_{R}}\left(S^{N}\left(B^{N}(Rt)\right)-
\mu^{R}B^{N}(Rt)\right)-\theta t,
\nonumber\\
I^{R}(t)&=&\frac{\mu^{R}}{d_{R}}Y^{N}(Rt)
=\frac{R\mu^{R}}{d_{R}}\int_{0}^{t}I\{\tilde{Q}^{R}(s)=0\}ds.
\nonumber
\end{eqnarray}
The process $I^{R}(\cdot)$ is non-decreasing process
and can increase only when the queue length process
$\tilde{Q}^{R}(\cdot)$ reaches zero due to the non-idling service
discipline.

\begin{lemma}\label{barheavytraI}
For each $N$ and $R$ and under conditions~\eq{fgc} and~\eq{heavconI},
we have, as $R\rightarrow\infty$, $X^{R}(\cdot)$ converges
weakly to a process $\tilde{X}(\cdot)$, that is,
\begin{eqnarray}
&&X^{R}(\cdot)\Rightarrow\tilde{X}(\cdot) =\lambda\pi
B_{H}(\cdot)-\theta \cdot \elabel{xrconverge}
\end{eqnarray}
where $\pi$ and $B_{H}(\cdot)$ are given in Theorem~\ref{heavytraI}.
\end{lemma}
\begin{proof}
Due to Lemma~\ref{equivalent}, it suffices to prove the below
facts, as $R\rightarrow\infty$,
\begin{eqnarray}
&&\frac{1}{d_{R}}\bar{A}\left(T^{N}(R\cdot)\right)
-\frac{1}{d_{R}}\mu^{R}R\cdot
\elabel{centerBA}\\
&=&\frac{1}{d_{R}} \left(\bar{A}\left(T^{N}(R\cdot)\right) -\lambda
T^{N}(R\cdot)\right)
+\frac{1}{d_{R}}\lambda\left(T^{N}(R\cdot)-\gamma NR\cdot\right)
-\theta\cdot
\nonumber\\
&\Rightarrow&\lambda\pi B_{H}(\cdot)-\theta \cdot, \nonumber
\end{eqnarray}
and
\begin{eqnarray}
&&\frac{1}{d_{R}}\left(S^{N}\left(B^{N}(R\cdot)\right)-
\mu^{R}B^{N}(R\cdot)\right) \Rightarrow 0. \elabel{centerSB}
\end{eqnarray}

As a matter of fact, notice that from the proof of Theorem 1 in
Taqqu {\em et al.}~\cite{taqwil:profun}, we know that the process
$T_{n}(\cdot)$ defined in \eq{tntime} has variance
\begin{eqnarray}
&&Var(T_{n}(t))\sim \pi^{2}t^{3-\alpha_{min}}L(t)\;\;\;\mbox{as}
\;\;\;t\rightarrow\infty. \nonumber
\end{eqnarray}
Then it follows from condition~\eq{fgc} and a similar proof as used
in justifying Theorem 4 in Mikosch {\em et al.}~\cite{mikres:nettra}
that the below weak convergence in the space $C[0,\infty)$ is true,
\begin{eqnarray}
&&\tilde{T}^{R}(\cdot)\equiv \frac{1}{d_{R}}(T^{N}(R\cdot)-\gamma
NR\cdot) \Rightarrow \pi B_{H}(\cdot)\;\;\;\mbox{as}\;\;\;
R\rightarrow\infty, \elabel{fraccon}
\end{eqnarray}
where $B_{H}$ is standard fractional Brownian motion with
$H=(3-\alpha_{min})/2$ and $\pi$ is given in \eq{pisquareI}
and \eq{pisquareII}.

Next, by Functional Central Limit Theorem (see, for example, Chen
and Yao~\cite{cheyao:funque}), we have that
\begin{eqnarray}
&&\tilde{\bar{A}}^{R}(\cdot)\equiv
\frac{1}{(NR)^{1/2}}\left(\bar{A}(NR\cdot)-\lambda NR\cdot\right)
\Rightarrow\xi^{a}(\cdot)\;\;\;\mbox{as}\;\;R\rightarrow\infty,
\elabel{abarwkc}
\end{eqnarray}
where the weak convergence is in the Skorohod topology and
$\xi^{a}(\cdot)$ is a Brownian motion with mean zero and variance
$\lambda$.

Moreover, for each $t\geq 0$, we have,
\begin{eqnarray}
\tilde{S}^{R}(t)
&\equiv&\frac{1}{(NR)^{1/2}}\left(S^{N}(Rt)-\mu^{R}Rt\right)
\nonumber\\
&=&
\frac{1}{(NR)^{1/2}}\left(\mbox{sup}\left\{k:
v(1)+...+v(k)\leq \mu^{R}Rt\right\}-\mu^{R}Rt\right)
\nonumber\\
&=&
\frac{1}{(NR)^{1/2}}\left(S_{1}^{N}(NR\mu_{1}^{R}t)
-NR\mu_{1}^{R}t\right),
\nonumber
\end{eqnarray}
where in the last equation, $\mu_{1}^{R}$ is given by
\begin{eqnarray}
&&\mu_{1}^{R}=\lambda\gamma+(R^{1-\alpha_{min}}L(R))^{1/2}\theta,
\nonumber
\end{eqnarray}
and $S_{1}^{N}(\cdot)$ is the counting process corresponding to
the i.i.d. normalized random sequence $\{v(i),i\geq 1\}$ with
mean 1. Moreover, by Lemma~\ref{uvwlimit}, we have that
$\mu^{R}_{1}\rightarrow\lambda\gamma$ as $R\rightarrow\infty$.
Then by Functional Central Limit Theorem, we have
\begin{eqnarray}
&&\tilde{S}^{R}(\cdot)\Rightarrow\xi^{s}(\lambda\gamma\cdot)\;\;\;
\mbox{as}\;\;\;R\rightarrow\infty, \elabel{tilsr}
\end{eqnarray}
where $\xi^{s}(\lambda\gamma\cdot)$ is a Brownian motion with
mean zero and variance $\lambda\gamma\sigma_{v}^{2}$.

Now notice the independent assumption among the processes
$\tilde{\bar{A}}^{R}(\cdot)$, $\tilde{S}^{R}(\cdot)$ and
$\tilde{T}^{R}(\cdot)$, and the properties that Brownian motion and
fractional Brownian motion have a.s. continuous sample paths, then
by Skorohod representation theorem (see, for example, Ethier and
Kurtz~\cite{ethkur:marpro}), we can and will assume that the
convergence in \eq{fraccon}-\eq{tilsr} is u.o.c. Thus, by
Lemma~\ref{uvwlimit} and for each $t\geq 0$, as
$R\rightarrow\infty$,
\begin{eqnarray}
\left\|\frac{1}{NR}\left(T^{N}(R\cdot)-\gamma NR\cdot\right)\right\|_{t}
&=&\frac{1}{N^{1/2}V(R)}
\left\|\frac{1}{d_{R}}\left(T^{N}(R\cdot)-\gamma NR\cdot\right)\right\|_{t}
\nonumber\\
&\rightarrow& 0\;\;\;\;\;\mbox{a.s.}, \nonumber
\end{eqnarray}
which implies that as $R\rightarrow\infty$,
\begin{eqnarray}
&&T_{1}^{R}(\cdot)\equiv\frac{1}{NR}T^{N}(R\cdot)\rightarrow\gamma\cdot
\;\;\;\;\;\mbox{a.s.}\;\;\;\mbox{u.o.c.} \elabel{rfluidI}
\end{eqnarray}
Therefore, by \eq{abarwkc}, \eq{rfluidI}, Random Change of Time
Theorem in Billinsley~\cite{bil:conpro} and Lemma~\ref{uvwlimit}, we
have
\begin{eqnarray}
&&\frac{1}{d_{R}}
\left(\bar{A}\left(T^{N}(R\cdot)\right)-\lambda T^{N}(R\cdot)\right)
\elabel{rfluidII}\\
&&=\frac{1}{U(R)(NR)^{1/2}}
\left(\bar{A}^{N}(NRT^{R}_{1}(\cdot))-\lambda NRT^{R}_{1}(\cdot)\right)
\nonumber\\
&&\rightarrow 0\;\;\;\;\;\mbox{a.s.}\;\;\;\mbox{u.o.c.}
\nonumber
\end{eqnarray}

Next, notice that, for each $t\geq 0$,
\begin{eqnarray}
&&B^{R}_{1}(t)\equiv\frac{B^{N}(Rt)}{R}\leq t. \nonumber
\end{eqnarray}
Then, it follows from \eq{tilsr} and Lemma~\ref{uvwlimit} that
\begin{eqnarray}
&&\left\|\frac{1}{d_{R}}\left(S^{N}\left(B^{N}(Rt)\right)-
\mu^{R}B^{N}(Rt)\right)\right\|_{t}
\nonumber\\
&&=\frac{1}{U(R)}\left\|\frac{1}{(NR)^{1/2}}
\left(S^{N}\left(RB^{N}_{1}(\cdot)\right)-
\mu^{R}RB^{N}_{1}(\cdot)\right)\right\|_{t}
\nonumber\\
&&\leq\frac{1}{U(R)}\left\|\frac{1}{(NR)^{1/2}}
\left(S^{N}\left(R\cdot\right)-
\mu^{R}R\cdot\right)\right\|_{t}
\nonumber\\
&&\rightarrow 0\;\;\;\;\;\mbox{a.s.}
\nonumber
\end{eqnarray}
Thus, we have, as $R\rightarrow\infty$,
\begin{eqnarray}
&&\frac{1}{d_{R}}\left(S^{N}\left(\bar{B}^{N}(R\cdot)\right)-
\mu^{R}\bar{B}^{N}(R\cdot)\right)\rightarrow 0\;\;\;\;\;
\mbox{a.s.}\;\;\;\mbox{u.o.c.} \elabel{rfluidIII}
\end{eqnarray}
Hence by \eq{fraccon}, \eq{rfluidII} and \eq{rfluidIII}, as
$R\rightarrow\infty$, the convergence stated in \eq{centerBA}
and \eq{centerSB} is true.
$\Box$
\end{proof}

\vskip 0.10in
The remaining proof of Theorem~\ref{heavytraI} is similar to that
used in justifying Theorem~\ref{heavytra}. Hence we omit it here.

\end{document}